\documentclass[12pt]{amsart}
\usepackage{amsfonts,amssymb,amsmath,mathtools}
\usepackage[margin=1.2in]{geometry}

\theoremstyle{plain}
\newtheorem{theorem}{Theorem}
\newtheorem{corollary}{Corollary}
\newtheorem{lemma}{Lemma}
\newtheorem{proposition}{Proposition}

\theoremstyle{definition}
\newtheorem{definition}{Definition}
\newtheorem{remark}{Remark}
\newtheorem{example}{Example}

\numberwithin{equation}{section}

\begin{document}

\title
{Multiresolution expansions and wavelets in Gelfand-Shilov spaces}

\author[S. Pilipovi\'{c}]{Stevan Pilipovi\'{c}}
\address{S. Pilipovi\'{c}\\ University of Novi Sad, Faculty of Sciences\\ Department
 of Mathematics and Informatics\\ Trg Dositeja Obradovi\' ca 4\\ 21000 Novi Sad \\ Serbia }
\email{pilipovics@yahoo.com}
\thanks{S. Pilipovi\'{c} and N. Teofanov were supported by Serbian Ministry of Education and Science through Project 174024, project 19/6-020/961-47/18 MNRVOID of the Republic of Srpska and ANACRES.}

\author[D. Raki\'{c}]{Du\v{s}an Raki\'{c}}
\address{D. Raki\'{c}\\ University of Novi Sad\\ Faculty of Technology\\ Bul. cara Lazara 1\\ 21000 Novi Sad \\ Serbia}
\email{drakic@tf.uns.ac.rs}
\thanks{ D. Raki\'{c} was supported by
Serbian Ministry of Education and Science through Project III44006 and by Provincial Secretariat for Higher Education and Scientific Research through Project 142-451-2102/2019.}

\author[N. Teofanov]{Nenad Teofanov}
\address{N. Teofanov\\ University of Novi Sad, Faculty of Sciences\\ Department
 of Mathematics and Informatics\\ Trg Dositeja Obradovi\' ca 4\\ 21000 Novi Sad \\ Serbia }
\email{nenad.teofanov@dmi.uns.ac.rs}

\author[J. Vindas]{Jasson Vindas}
\address{J. Vindas\\ Department of Mathematics: Analysis, Logic and Discrete Mathematics\\ Ghent University\\ Krijgslaan 281\\ B 9000 Gent\\ Belgium}
\email{jasson.vindas@UGent.be}
\thanks{J. Vindas was supported by Ghent University through the BOF-grants 01J11615 and 01J04017.}


\begin{abstract}
We study approximation properties generated by highly regular
scaling functions and orthonormal wavelets. These properties are
conveniently described in the framework of Gelfand-Shilov spaces.
Important examples of multiresolution analyses for which our
results apply arise in particular from Dziuba\'nski-Hern\'andez
construction of band-limited wavelets with subexponential decay.
Our results are twofold. Firstly, we obtain approximation
properties of multiresolution expansions of Gelfand-Shilov
functions and (ultra)distributions. Secondly, we establish
convergence of wavelet series expansions in the same regularity
framework.

\end{abstract}

\maketitle

\section{Introduction}
\label{intro}

Multiresolution analysis and orthogonal wavelet bases are very powerful tools in several areas of mathematics and its applications. They are particularly useful in approximation theory  and their approximation properties have been extensively studied  in many function and distribution spaces \cite{HW,Mey}. These approximation features are intimately connected with the regularity properties of the scaling functions and wavelets, where regularity is usually quantified by how smooth they are and how fast they (and their derivatives) decay at infinity. As it is well-documented in the literature, there is however a trade-off between the latter two properties. For instance, it is
well-known that smooth orthonormal wavelets, with all derivatives
bounded,  cannot have exponential decay \cite[Corollary 5.5.3]{Dau}.

In \cite{DzH} Dziuba\'{n}ski and Hern\'{a}ndez constructed an orthonormal wavelet with \emph{subexponential} decay; see also \cite{FKU} for other orthonormal wavelets having that decay. Their construction has been generalized by various other authors \cite{MT,pathak-s2008}. The Dziuba\'{n}ski-Hern\'{a}ndez wavelets are of Lemari\'{e}-Meyer type (i.e., band-limited), hence, they are entire functions. They are thus also MRA wavelets \cite{HW}. We have observed here that they as well as corresponding scaling functions belong to Gelfand-Shilov spaces (see Section \ref{Sec1} for definitions and basic properties). It is therefore natural to investigate their approximation properties in such spaces. The Gelfand-Shilov spaces were introduced in the context of parabolic initial-value problems \cite{GS}, and have been widely used afterwards
in situations when both (sub/super)exponential type decay and extension to the complex domain are important, e.g., in the study of traveling waves. More details about applications of  the Gelfand-Shilov spaces can be found in \cite{Gram,NR} and references therein.

In this article we introduce the notions of $(\rho_1,\rho_2)$-regular MRA and orthonormal wavelets, where the parameters $\rho_1$ and $\rho_2$ measure the Gelfand-Shilov regularity of the scaling functions and wavelets. Our goal is to study their effectiveness to approximate functions and (ultra)distributions in Gelfand-Shilov spaces. As we explain in Section \ref{Sec MRA-wavelet}, the Dziuba\'{n}ski-Hern\'{a}ndez wavelets are particular examples of the MRA and orthonormal wavelets covered by our considerations.

Our main results are Theorem \ref{th conv MRA rho regular} and Theorem \ref{th convergence}. Interestingly, although their proofs are quite different from each other, there is a loss of regularity phenomenon in both main results that is quantified by the same parameter. Theorem \ref{th conv MRA rho regular} deals with the convergence of multiresolution expansions in Gelfand-Shilov spaces. Its proof is based on a refinement of Meyer's powerful method \cite[Section 2.6]{Mey}. Our adaptation of the method involves several subtle points related to the nature of Gelfand-Shilov spaces  which we believe are interesting in its own right. Theorem \ref{th convergence} establishes convergence properties of wavelet series in subspaces of  Gelfand-Shilov spaces consisting of functions for which all moments vanish. Our approach in Section \ref{SecWaveletGS} partly relies on continuity properties of the wavelet transform, obtained by the authors in \cite{PRTV}.

We end this introduction with a few words related to the vanishing moment conditions imposed in Section \ref{SecWaveletGS}. First of all, notice that if an orthonormal wavelet belongs to the Schwartz space $\mathcal{S}(\mathbb{R}^{d})$, then all of its moments
must vanish \cite{HW}. In \cite{hol1} Holschneider develops a distributional theory for the continuous wavelet transform based on continuity theorems for the Lizorkin spaces (cf. \cite{PRV}), which has been used in \cite{SanVin} to establish the convergence of wavelet series in the same spaces. This distributional theory has important implications in Tauberian theory \cite{pilipovic-vindas2019}. An ultradistributional framework for wavelet transforms is provided in \cite{PRTV} (cf. \cite{FKY}).


\subsection{Notation}The notation $ A\lesssim B $ means that $ A
\leq C \cdot B $ for some positive implicit constant $ C$. The dual pairing
between a test function space $ {\mathcal A}$ and its dual
${\mathcal A'}$ is denoted by $\langle \cdot, \cdot\rangle
=\:_{\mathcal A'}\langle \cdot, \cdot
\rangle_{\mathcal A} $.

\section{Gelfand-Shilov type spaces and the wavelet transform}  \label{Sec1}
We collect in this preliminary section some facts about Gelfand-Shilov spaces and mapping properties of wavelet transforms on these spaces.

Let $ {\rho_1}, {\rho_2} \geq 0.$ A function $ \varphi \in \mathcal{S}(\mathbb{R}^d) $ belongs to
the Gelfand-Shilov space $ \mathcal{S}_{{\rho_2}}^{{\rho_1}}
(\mathbb{R}^d)$  if  there exists  a
constant $  h > 0 $ such that
$$
|x^{\alpha} \varphi^{(\beta)} (x)| \lesssim h^{-|\alpha + \beta|}
\, \alpha!^{{\rho_2}} \beta!^{{\rho_1}}, \qquad  x \in \mathbb{R}^d,
\, \alpha, \beta \in \mathbb{N}^d,
$$
with implicit constant being independent of $ \alpha $ and $ \beta. $ The space $ \mathcal{S}_{{\rho_2}}^{{\rho_1}} (\mathbb{R}^d)$ is
nontrivial \cite{GS} if and only if $ {\rho_1} + {\rho_2} > 1 $ or  $\rho_1+\rho_2=1$ and $\rho_1,\rho_2>0$. We note that if $\rho_{1}=1$ the elements of  $ \mathcal{S}_{{\rho_2}}^{{\rho_1}} (\mathbb{R}^d)$ are real analytic functions, and if $\rho_1<1$ they are even entire functions.
The
family of norms
\begin{equation*} 
 p_h^{{\rho_1}, {\rho_2}} (\varphi) = \sup_{x \in \mathbb{R}^d, \alpha,
\beta \in \mathbb{N}^d} \frac{h^{|\alpha +
\beta|}}{\alpha!^{{\rho_2}} \beta!^{{\rho_1}}} \, |
x^{\alpha}\partial^{\beta}\varphi(x)|, \qquad h > 0,
\end{equation*}
defines the canonical inductive limit topology of $
\mathcal{S}_{{\rho_2}}^{{\rho_1}} (\mathbb{R}^d), $ and it becomes a
(DFS)-space with this topology \cite[Proposition 3.5]{DV} (in fact, one can show it is a (DFN)-space \cite[Proposition 2.11]{PPV}). Consequently, $
\mathcal{S}_{{\rho_2}}^{{\rho_1}} (\mathbb{R}^d)$ is a Montel space \cite[Corollary A.5.11]{Mor}. Observe that the
family of norms
\begin{equation} \label{norma-GS2}
p_{h, c}^{{\rho_1}, {\rho_2}} (\varphi) = \sup_{x \in
\mathbb{R}^d, \beta \in \mathbb{N}^d}
\frac{h^{|\beta|}}{\beta!^{{\rho_1}}} \, e^{c
|x|^{\frac{1}{\rho_2}}} |\partial^{\beta}\varphi(x)|, \qquad h, c > 0,
\end{equation}
induces an equivalent topology on $ \mathcal{S}_{{\rho_2}}^{{\rho_1}} (\mathbb{R}^d) $. Also, the Fourier transform $\mathcal{F}$ is an isomorphism between $
\mathcal{S}_{{\rho_2}}^{{\rho_1}} (\mathbb{R}^d) $ and $
\mathcal{S} ^{{\rho_2}} _{{\rho_1}} (\mathbb{R}^d). $
 We denote by $ \mathcal{D}^{\rho_1} (\mathbb{R}^d) $ the subspace of $ \mathcal{S}_{{\rho_2}}^{{\rho_1}}
(\mathbb{R}^d)$ consisting of compactly supported Gevrey ultradifferentiable functions (by the Denjoy-Carlemann theorem \cite{Rudin}, it is non-trivial if and only if $ \rho_1 > 1 $). Note then \cite{GS} that $ \mathcal{S}_{0}^{{\rho_1}}
(\mathbb{R}^d) =\mathcal{D}^{\rho_1} (\mathbb{R}^d) $ and $ \mathcal{S}_{\rho_2}^{0}
(\mathbb{R}^d) =\mathcal{F}(\mathcal{D}^{\rho_2} (\mathbb{R}^d) )$.

We are interested in the wavelet transform in the context of Gelfand-Shilov spaces. We use the notation $ \mathbb{H}^{d + 1} = \mathbb{R}^d \times\mathbb{R}_+. $ In general we call an integrable function $\psi$ a wavelet if $
\int_{\mathbb{R}^{d}} \psi (x) dx = 0$. When $ \psi \in {\mathcal S}^{\rho_1}_{\rho_2}
(\mathbb{R}^d)$ one can define the wavelet transform of an
ultradistribution $ f \in ({\mathcal S}^{\rho_1}_{\rho_2}
(\mathbb{R}^d))' $ with respect to the wavelet $ \psi$ as
\begin{equation}  \label{Wavelet}
\mathcal{W}_{\psi} f (b, a) = \left\langle f (x), \frac{1}{a^d}
\bar{\psi}\left(\frac{x - b}{a}\right) \right\rangle =
\frac{1}{a^d}\int_{\mathbb{R}^d} f (x)  \bar{\psi}\left(\frac{x -
b}{a}\right) \,\mathrm{d}x,
\end{equation}
where $(b,a)\in \mathbb{H}^{d+1}$. This definition of course extends to $
\psi \in {\mathcal A} $ and $ f \in {\mathcal A}' $ whenever the
dual pairing in \eqref{Wavelet} makes sense. In particular, for
$ \psi,f \in L^2 (\mathbb{R}^d)$, the wavelet transform of $f \in  L^2 (\mathbb{R}^d)$ is given by the integral formula in \eqref{Wavelet}; the same applies for a measurable function $f$ satisfying bounds  $|f(x)|\lesssim e^{c|x|^{1/\rho_2}}$, for each $c>0$, when the wavelet $ \psi \in {\mathcal S}^{\rho_1}_{\rho_2}(\mathbb{R}^d)$.
The next proposition\footnote{The result is stated in \cite{PRTV} under the extra assumption $\rho_1>0$, but the proof given there actually works for $\rho_1=0$ as well. The same comment applies to Theorem \ref{th3}.} estimates the order of growth of the wavelet transform of an ultradistribution.

 \begin{proposition}[{\cite[Proposition 3 and Remark 4 ]{PRTV}}] \label{Calderon 2}
Let $ s>
{\rho_1}\geq 0$ and $ t>{\rho_2}>0$. If $ \psi \in {\mathcal
S}^{{\rho_1}}_{{\rho_2}} (\mathbb{R}^d) $ and $ B $ is a bounded set in $ ({\mathcal
S}^s_t (\mathbb{R}^d))' , $ then for  each $ k > 0$,
$$
|{\mathcal W}_{\psi} f(b, a)| \lesssim
 e^{k \big( a^{\frac{1}{t - {\rho_2}}} + (\frac{1}{a})^{\frac{1}{s - {\rho_1}}}
+ |b|^{\frac{1}{t}}  \big)}, \qquad (b,a) \in \mathbb{H}^{d+1},
$$
uniformly for $ f \in B. $
\end{proposition}

In order to study further mapping properties of the wavelet transform on Gelfand-Shilov spaces, we have introduced in \cite{PRTV} a number of spaces of highly localized functions on $\mathbb{R}^{d}$, which are Gelfand-Shilov analogs of those considered in Holschneider's approach to the wavelet transform in Schwartz' spaces \cite{hol1} (cf. \cite{PRV}).  We define $ (\mathcal{S}_{{\rho_2} } ^{{\rho_1}})_0
(\mathbb{R}^d) $ as the closed subspace of $
\mathcal{S}_{{\rho_2}}^{{\rho_1}} (\mathbb{R}^d)$
consisting of $\varphi\in \mathcal{S}_{{\rho_2}}^{{\rho_1}} (\mathbb{R}^d)$
whose moments of any order vanish, that is,
\begin{equation*} 
(\mathcal{S}_{{\rho_2} }^{{\rho_1}})_0
(\mathbb{R}^d)=\left\{\varphi \in
\mathcal{S}_{{\rho_2}}^{{\rho_1}} (\mathbb{R}^d): \int_{\mathbb{R}^{d}}x^{\alpha}\varphi(x)\mathrm{d}x = 0,\ \forall \alpha \in \mathbb{N}^{d} \right\}.
\end{equation*}
The Denjoy-Carleman theorem implies that $ (\mathcal{S}_{{\rho_2}}^{{\rho_1}})_0
(\mathbb{R}^d) $ is non-trivial if and only if $ \rho_2 > 1$. The following Gelfand-Shilov type spaces in the upper half-plane are appropriate to study the range of the wavelet transform in this context. Let $ s, t, \tau_1, \tau_2 > 0 $. A smooth function $ \Phi $
belongs to $ \mathcal{S}^{s}_{t, \tau_1, \tau_2}
(\mathbb{H}^{d+1})$  if for every $\alpha \in \mathbb{N} $ there
exists a constant $h>0$ such that
\begin{equation} \label{norma-GSupper}
p_{\alpha, h}^{s, t, \tau_1, \tau_2} (\Phi) = \sup_{((b, a),
\beta) \in \mathbb{H}^{d+1} \times \mathbb{N}^d}
\frac{h^{|\beta|}}{\beta!^{s}} \, e^{h\left(a^{1/\tau_1} +
a^{-1/\tau_2} + |b|^{1/t} \right)} \, \left|\partial_a^{\alpha}
\partial_b^{\beta} \Phi (b, a) \right|.
\end{equation}
The topology of $ \mathcal{S}^{s}_{t,\tau_1, \tau_2}
(\mathbb{H}^{d+1}) $ is defined via the family of seminorms
\eqref{norma-GSupper}, as inductive limit with respect to $h$ and
projective limit with respect to $\alpha$.

\begin{theorem} [{\cite[Theorem 1]{PRTV}}]\label{th3} Let  $\rho_1 \geq 0 $, $ \rho_2 > 1 $ and let
$ s >0$, $ t > \rho_1 + \rho_2 $,  $ \tau_1 >\rho_1  $ and $
\tau_2  > \rho_2 - 1$. Then the wavelet mapping
$$
{\mathcal W} : ({\mathcal S}^{\rho_1} _ {\rho_2})_0 (\mathbb{R}^d)
\times ({\mathcal S}^{\min\{s, \tau_2 - \rho_2 +1\}}_{1 - \rho_1+
\min\{t - \rho_2, \tau_1\} })_0 (\mathbb{R}^d) \to {\mathcal
S}^{s}_{t, \tau_1, \tau_2} (\mathbb{H}^{d+1}),
$$
given by $ {\mathcal W} : (\psi, \varphi) \mapsto {\mathcal
W}_{\psi} \varphi $, is continuous.
\end{theorem}

\section{$(\rho_1,\rho_2)$-Regular MRA and wavelets}
\label{Sec MRA-wavelet}

In this section we introduce highly regular multiresolution analysis and wavelets. We will show  in the subsequent sections that they enjoy very good approximation features in context of Gelfand-Shilov spaces.  For the reader's convenience, we recall \cite{HW,Mey} that a multiresolution approximation (MRA) is
an increasing sequence $ \{V_m\}_{m \in \mathbb{Z}} $ of closed
linear subspaces of $ L^2 (\mathbb{R}^d) $ with the following four
properties:
\begin{itemize}
\item[(i)] $  \bigcap_{m\in\mathbb{Z} } {V_m} = \{0\} $ and $ \bigcup_{m \in\mathbb{Z}} {V_m} $ is dense in $ L^2
(\mathbb{R}^d); $
\item[(ii)] $ f(x) \in V_m \Leftrightarrow f(2x) \in V_{m + 1}, \; m \in \mathbb{Z}; $
\item[(iii)]  $ f(x) \in V_0 \Leftrightarrow f(x - n) \in V_0, \;  \, n \in \mathbb{Z}^d; $
\item[(iv)] there exists $ \phi \in L^{2}(\mathbb{R}^{d}) $ such that $ \{\phi(x
- n)\}_{n \in \mathbb{Z}^d} $ is an orthonormal basis of $ V_0. $
\end{itemize}
The function $\phi$ from (iv) is called a scaling function for the given MRA.

It is well known that the effectiveness of an MRA to approximate functions depends on how regular a scaling function could be chosen inside $V_0$. In this regard Meyer \cite{Mey} introduced the notion of $r$-regular MRA and studied approximation properties in subspaces of $\mathcal{S}'(\mathbb{R}^{d})$, see \cite{KosVin,PilTeo} for generalizations. The ensuing finer regularity notion is well-suited for our analysis.

\begin{definition}\label{def MRA rho regular} Let $\rho_1\geq 0$ and $\rho_2>1$. An MRA  is called $(\rho_1,\rho_2)$-regular if it possesses a scaling function $\phi\in \mathcal{S}^{\rho_1}_{\rho_2}(\mathbb{R}^{d}).$
\end{definition}

Whenever we speak about a  $(\rho_1,\rho_2)$-regular MRA, we implicitly fix a scaling function and always assume that $\phi\in \mathcal{S}^{\rho_1}_{\rho_2}(\mathbb{R}^{d})$. It is very important to point out that the condition $\rho_2>1$ in Definition \ref{def MRA rho regular} is dictated by the nature of an MRA. Indeed, we prove below in Remark \ref{rknoMRArho_2<1} that if $\rho_{2}\leq 1$, there is no MRA with scaling function $\phi\in \mathcal{S}_{\rho_2}^{\rho_1}(\mathbb{R}^{d})$. On the other hand, Example \ref{onbDzH} implies the existence of  $(\rho_1,\rho_2)$-regular MRA with the parameter constrains from our definition.

We also recall that a function $ \psi \in L^2 (\mathbb{R}) $ is called an {\em
orthonormal wavelet} if the set $ \{\psi_{m, n} : m \in
\mathbb{Z}, n \in \mathbb{Z}\} $ is an orthonormal basis of $ L^2
(\mathbb{R}),$ where
$$
\psi_{m, n}(x) = 2^{\frac{m}{2}} \psi(2^m
x - n), \qquad m,n \in \mathbb{Z},\  x \in \mathbb{R}.
$$
Note that if additionally $\psi \in L^{1}(\mathbb{R})$, then \cite[Theorem~3.3.1, p.~63]{Dau} the orthonormal wavelet necessarily satisfies $ \int_{\mathbb{R}} \psi (x) dx = 0$.

A powerful way to construct orthonormal wavelets is via MRA. The procedure how one can assign a wavelet to an MRA is very well understood and explained in the literature. Instead of giving any detail, we refer for instance to the book \cite{HW} by Hern\'{a}ndez and Weiss for an excellent account on the subject. If an orthonormal wavelet is associated to an MRA, we simply call it an \emph{MRA wavelet}. We mention that any band-limited orthonormal wavelet with continuous Fourier transform or any orthonormal wavelet belonging to  $\mathcal{S}(\mathbb{R})$ are MRA wavelets (cf. \cite[Corollary 3.16, p. 363]{HW}). Also, it is well-known \cite{HW,Mey} that  an orthonormal wavelet  $\psi\in\mathcal{S}(\mathbb{R})$ must automatically belong to $\mathcal{S}_{0}(\mathbb{R})$, the subspace of $\mathcal{S}(\mathbb{R})$ consisting of functions whose all moments vanish.

\begin{definition} An orthonormal wavelet $\psi$ is called $(\rho_1,\rho_2)$-regular if $\psi\in (
\mathcal{S}_{{\rho_2}}^{{\rho_1}}) (\mathbb{R})$ and if it arises from a $(\rho_1,\rho_2)$-regular MRA.
\end{definition}

It immediately follows that all moments of a $(\rho_1,\rho_2)$-regular orthonormal
wavelet $\psi$ vanish, namely,  $\psi\in (
\mathcal{S}_{{\rho_2}}^{{\rho_1}})_{0} (\mathbb{R})$.

The next example discusses the existence of $(\rho_1,\rho_2)$-regular MRA and wavelets. In fact, it turns out that the Dziuba\'nski-Hern\'andez
orthonormal wavelets, constructed in \cite{DzH}, are $(\rho_1,\rho_2)$-regular. Note that these kinds of wavelets are in particular of Lemari\'{e}-Meyer type \cite{HW,LM}. Tensorizing their associated one dimensional scaling functions, one readily obtains examples of $(\rho_1,\rho_2)$-regular MRA on $\mathbb{R}^{d}$.


\begin{example} \label{onbDzH}
Let  $ \rho_2 > 1. $ We consider here  Dziuba\'nski-Hern\'andez
orthonormal wavelets, which are constructed as follows. Fix $ a < \pi/3 $. Pick $ \varphi \in \mathcal{D}^{\rho_2} (\mathbb{R}) $ such that $ \mbox{supp }\varphi \subseteq [-a, a] $ and $ \int_{-\infty}^{\infty} \varphi (\xi) \, d\xi = \pi/2 $ (which exists in view of the Denjoy-Carlemann theorem).  One also considers $\varphi_2(\xi)=(1/2) \varphi(\xi/2)$ and the bell type function given by
$$
b(\xi) = \sin \Big ( \int_{-\infty} ^{\xi - \pi} \varphi (t)
dt \Big ) \cos \Big ( \int_{-\infty} ^{\xi - 2\pi} \varphi_{2} (t)
dt \Big )
$$
for $\xi>0$, and extended evenly to $ (-\infty, 0]$.  It is verified in \cite{DzH} that  $ \mbox{supp
}b \subseteq [-8 \pi/3, -2 \pi/3] \cup [2 \pi/3, 8 \pi/3]$ and that $b\in\mathcal{D}^{\rho_2}(\mathbb{R})$. Finally, the associated Lemari\'{e}-Meyer orthonormal wavelet to $b$ \cite{HW}, given in Fourier side as
\begin{equation} \label{dziu-her-wavelet}
\hat \psi (\xi) = e^{i \xi /2} b (\xi), \;\;\; \xi \in
\mathbb{R},
\end{equation}
satisfies $\psi\in \mathcal{F}(\mathcal{D}^{\rho_2}(\mathbb{R}))\subseteq \mathcal{S}^{\rho_1}_{\rho_2}(\mathbb{R})$ for all $\rho_1\geq 0$.

It is proved in \cite{HW} (see also \cite{LM,DzH}) that
$ \psi $ given by  \eqref{dziu-her-wavelet} is an orthonormal
wavelet. We have that $ \hat{\psi} \in \mathcal{D}^{\rho_2}
(\mathbb{R}) \subseteq \mathcal{S}_{\rho_1}^{\rho_2} (\mathbb{R}) $
and since $ 0 \notin \mbox{supp }\hat{\psi}, $ we obtain $ \psi \in
(\mathcal{S}^{\rho_1}_{\rho_2} )_{0}(\mathbb{R}). $ As mentioned above, it follows that $ \psi $ is an MRA wavelet, and there are several ways \cite{LM} to select an associated scaling function in $\mathcal{S}(\mathbb{R})$.
For example, we can choose

\begin{equation*}
|\hat \phi (\xi)|^2 = \left \{
\begin{array}{lll}
1 & \text{if} & |\xi| \leq 2\pi/3, \\
b^2 (2\xi) & \text{if} & 2\pi/3 \leq |\xi| \leq 4\pi/3, \\
0 & \text{if} & |\xi| \geq 4\pi/3, \;\;\; \xi \in \mathbb{R},
\end{array}
\right.
\end{equation*}
and take $\arg \hat{\phi}(\xi)=\xi$, that is, the smooth function $ \hat{\phi} (\xi) = e^{i \xi} |\hat{\phi} (\xi)|, $
see also \cite[Theorem 4.12, page 133]{HW}. Since the function $b\in\mathcal{D}^{\rho_2}(\mathbb{R})$,
we conclude that $\hat{\phi}\in\mathcal{D}^{\rho_2}(\mathbb{R})$ and therefore the scaling function of this MRA also satisfies $\phi\in\mathcal{S}^{\rho_1}_{\rho_2}(\mathbb{R})$ for any $\rho_1\geq 0$.

\end{example}

\begin{remark}
\label{rknoMRArho_2<1} Let $\rho_{2}\leq 1$. We show that there is no MRA with scaling function $\phi\in \mathcal{S}_{\rho_2}^{\rho_1}(\mathbb{R}^{d})$. Indeed, as soon as a scaling function $\phi\in \mathcal{S}_{\rho_2}^{\rho_1}(\mathbb{R}^{d})\subset\mathcal{S}(\mathbb{R}^{d})$, one can readily verify that the function $\kappa(x)=\sum_{k\in\mathbb{Z}^{d}} \bar{\phi}(-k)\phi(x-k)$  belongs to $\mathcal{S}_{\rho_2}^{\rho_1}(\mathbb{R}^{d})$ (see e.g. Lemma \ref{projE} below). The Fourier transform of $\kappa$ automatically satisfies $\hat\kappa(0)=1$ and $\partial^{\alpha}\hat{\kappa}(0)=0$ for any nonzero multi-index $\alpha$, as follows from \cite[Theorem~4, p.~33]{Mey}. If $\rho_2$ were smaller than or equal to 1,  the function $\hat{\kappa}\in \mathcal{S}^{\rho_2}_{\rho_1}(\mathbb{R}^{d})$ would be identically equal to  1 due to analyticity of the elements of $\mathcal{S}^{\rho_2}_{\rho_1}(\mathbb{R}^{d})$ when $\rho_2\leq 1$, which is of course impossible.
\end{remark}

\section{Converge of multiresolution expansions in Gelfand-Shilov spaces} \label{Sec3}

The goal of this section is to study converge results in Gelfand-Shilov spaces for multiresolution expansions.  So, we now focus on approximation properties of $(\rho_1,\rho_2)$-regular MRA.

We need to fix some notation. Given an MRA $\{V_m\}_{m\in\mathbb{Z}}$, the orthogonal projection $L^{2}(\mathbb{R}^{d})\to V_0$ is denoted as $q_0$; it is in fact determined by its kernel
\begin{equation} \label{Eq:E_0}
q_0(x, y) = \sum_{k \in \mathbb{Z}^d} \phi(x - k) \bar{\phi}(y - k), \qquad x, y \in \mathbb{R}^d,
\end{equation}
i.e.,
$$
(q_0 f)(x) = \langle f(y), q_0(x, y) \rangle = \int_{\mathbb{R}^d}
f(y) q_0 (x, y) \, dy, \qquad x \in \mathbb{R}^d.
$$
Then, the orthogonal projection $ q_m: L^2(\mathbb{R}^d) \to V_m $ is given by
\begin{equation}
\label{q_mprojection}
(q_m f)(x) = \langle f(y), q_m (x, y) \rangle = \int_{\mathbb{R}^d}
f(y) q_m (x, y) \, dy, \qquad x \in \mathbb{R}^d,
\end{equation}
where
$
q_m (x, y) = 2^{m d} q_0(2^m x, 2^m y).
$
The sequence $ \{q_m f\}_{m \in \mathbb{Z}} $ is called the multiresolution expansion of $ f. $ We anticipate that Lemma~\ref{projE} below tells us that for each fixed $x$ the kernels of the orthogonal projections of a $(\rho_1,\rho_2)$-regular MRA satisfy  $q_m(x,\:\cdot\:)\in\mathcal{S}_{\rho_2}^{\rho_1}(\mathbb{R}^{d})$. This allows one to define each $q_mf$ even for an ultradistribution $f\in(\mathcal{S}_{\rho_2}^{\rho_1}(\mathbb{R}^{d}))'$ via the dual pairing in the formula \eqref{q_mprojection}.

Our main result concerning multiresolution expansions is the following theorem. Notice that there is a loss of regularity with respect to ultradifferentiability in this convergence result, which is quantitatively measured by the parameter $\sigma$ below. Interestingly, the same phenomenon shows up in Section \ref{SecWaveletGS} for convergence wavelet series or in the study of continuity of wavelet transforms in Gelfand-Shilov spaces \cite{PRTV}.

\begin{theorem} \label{th conv MRA rho regular} Let $ \rho_2>1$, $ \rho_1\geq 0,$  and let $\{V_m\}_{m\in\mathbb{Z}}$ be a $ (\rho_1, \rho_2)$-regular MRA. Set $\sigma=\rho_1+\rho_2-1$ and let $s\geq \sigma$ and $t\geq \rho_2$.
\begin{enumerate}
\item [(i)]
If  $ \varphi \in \mathcal{S}^{s - \sigma}_{t} (\mathbb{R}^d),$ then
$$
\lim_{m \to \infty} q_m \varphi = \varphi \; \mbox{ in }
\mathcal{S}^{s}_{t} (\mathbb{R}^d).
$$

\item [(ii)] If  $ f \in (\mathcal{S}^{s }_{t})' (\mathbb{R}^d),$ then
$$
\lim_{m \to \infty} q_m f = f \; \mbox{ in }
(\mathcal{S}^{s-\sigma}_{t} (\mathbb{R}^d))'.
$$
\end{enumerate}
\end{theorem}

The rest of the section is devoted to supplying a proof of Theorem \ref{th conv MRA rho regular}. In preparation, we need some auxiliary results. We closely follow Meyer's method as explained in \cite[pp. 39--41]{Mey}, but making the necessary adjustments for the context of Gelfand-Shilov spaces.

The next lemma is very simple but useful. It characterizes convergence of sequences in Gelfand-Shilov spaces.
\begin{lemma} \label{conv1} Let $ \{ \varphi_n \}_{n \in \mathbb{N}} $ be a sequence of
elements of $ \mathcal{S}^{s}_{t} (\mathbb{R}^d). $ Then, $ \{
\varphi_n \}_{n \in \mathbb{N}} $ converges in $
\mathcal{S}^{s}_{t} (\mathbb{R}^d) $ if and only if $ \{ \varphi_n
\}_{n \in \mathbb{N}} $ is bounded in $ \mathcal{S}^{s}_{t}
(\mathbb{R}^d)$ and it is pointwise convergent.
\end{lemma}

\begin{proof} The direct implication is of course trivial. Conversely, let $\varphi$ be the pointwise limit of the sequence.  As we have already pointed out in Section \ref{Sec1}, the space $ \mathcal{S}^{s}_{t} (\mathbb{R}^d)$ is Montel; in particular, it has the Heine-Borel property. Thus, being bounded, the set $\{\varphi_n: \: n \in \mathbb{N}\} $ is relatively compact in $ \mathcal{S}^{s}_{t} (\mathbb{R}^d)$. It therefore suffices to verify that the sequence has a unique accumulation point. But if for some subsequence $\varphi_{n_k}\to \eta $ in $ \mathcal{S}^{s}_{t} (\mathbb{R}^d)$, we must necessarily have $\eta=\varphi$, due to the pointwise convergence hypothesis. This completes the proof of the lemma.
\end{proof}

We need the following regularity property of the kernel $q_{0}(x,y)$.

\begin{lemma} \label{projE}
Given a $ (\rho_1, \rho_2)$-regular MRA, the reproducing kernel $q_0$ of $V_0$ satisfies the ensuing bounds: there exist constants $ c > 0 $ and  $ h > 0 $ such that
\begin{equation} \label{E-decay}
|\partial_x^{\alpha} \partial_y^{\beta} q_0(x, y)| \lesssim
h^{|\alpha + \beta|} \alpha!^{\rho_1} \beta!^{\rho_1}  \, e^{-c |x -
y|^{\frac{1}{\rho_2}}}, \qquad \alpha, \beta \in \mathbb{N}^d, \; x, y
\in \mathbb{R}^d.
\end{equation}
\end{lemma}

\begin{proof} Since $ \phi \in \mathcal{S}^{\rho_1}_{\rho_2} (\mathbb{R}^d), $ by \eqref{Eq:E_0},
we have that, for some $ h, c > 0,$
$$
|\partial_x^{\alpha} \partial_y^{\beta} q_0(x, y)| \lesssim
h^{|\alpha + \beta|} \alpha!^{\rho_1} \beta!^{\rho_1} \sum_{k \in \mathbb{Z}^{d}} e^{-c |x - k|^{\frac{1}{\rho_2}} -c |y - k|^{\frac{1}{\rho_2}}},
$$
for $ \alpha, \beta \in \mathbb{N}^d, \; x, y \in \mathbb{R}^d. $
 Now, since $ W(|x|) = e^{-c |x|^{1/\rho_2}} \in L^1 ([0, \infty)) $ is a decreasing function and $ W(0)<\infty,$ by
(a straightforward $d$-dimensional extension of) \cite[Lemma 3.12, p. 220]{HW} it follows that
$$
\sum_{k \in \mathbb{Z}^{d}} e^{-c |x - k|^{\frac{1}{\rho_2}} -c |y - k|^{\frac{1}{\rho_2}}} \lesssim e^{-\frac{c}{2} |x - y|^{\frac{1}{\rho_2}}},
$$
so that \eqref{E-decay} holds true.

\end{proof}

The ensuing lemma is the key to our adaptation of Meyer's method. It is an analog of \cite[Lemma~12, p.~40]{Mey} and provides a uniform estimate for primitives of functions with a number vanishing moments and satisfying a priori decay property.

\begin{lemma} \label{LemmaIntegrationMRA} Let $\rho>1$, $r\in\mathbb{N}$, and let $g\in C(\mathbb{R}^{d})$ be such that $\sup_{y\in\mathbb{R}^{d}}|g(y)|e^{c|y|^{1/\rho}}<\infty$, for some $c>0$, and $\int_{\mathbb{R}^{d}} y^{\alpha}g(y)dy=0$ for each $|\alpha|\leq r$. Then, there are absolute constants $C,h>0$ (only depending on the values of $c$, $\rho$, and the dimension) such that for any such $g$ and $r$ one can find functions $g_{\beta}\in C^{r}(\mathbb{R}^{d})$, $|\beta|=r$, for which

\begin{equation}\label{MRAeq int1} \displaystyle g = \sum_{|\beta| = r} \partial^{\beta} g_{\beta},
\end{equation}
$\int_{\mathbb{R}^{d}}g_{\beta}(y)dy=0$, and
\begin{equation}\label{MRAeq int2}
|g_{\beta} (y)| \leq C r!^{\rho-1} h^{r}
e^{-\frac{c }{2^{d}}|y|^{\frac{1}{\rho}}}\sup_{u\in\mathbb{R}^{d}}|g(u)|e^{c|u|^{\frac{1}{\rho}}}.
\end{equation}
\end{lemma}

\begin{proof} We will show the theorem by induction on the dimension $d$. Let us thus first assume that $ d=1$, so that $\beta = r\geq 1$.
Let
\begin{align*}
g_r (y)  = - \frac{1}{(r - 1)!} \int_y^{\infty} (y -
w)^{r - 1} g (w) \, dw = \frac{1}{(r - 1)!} \int_{-\infty}^y (y - w)^{r - 1} g
(w) \, dw,
\end{align*}
where the latter equality follows from the vanishing moment property of $g.$ Here above, and in what follows, we take $  -\int_y^{\infty}, $ when
$ y > 0, $ and $  \int_{-\infty}^y, $ when $ y < 0. $ Clearly, \eqref{MRAeq int1} holds and $\int_{-\infty}^{\infty}g_{r}(y)dy=0$, as follows from Fubini's theorem and $\int_{-\infty}^{\infty}w^{r}g(w)\,dw=0$. It remains to establish the inequality \eqref{MRAeq int2}.
Assume that $ x > 0, $ the case $ x < 0 $ is analogous. We set $M=\sup_{u\in\mathbb{R}}|g(u)|e^{c|u|^{1/\rho}}.$  Then, we have
$$
|g_r (x)|  \leq  \frac{M}{(r-1)!}
\int_0^{\infty} u^{r - 1} e^{-c (u +
x)^{\frac{1}{\rho}}} \, du
\leq \frac{M}{(r-1)!} e^{-\frac{c}{2} x^{\frac{1}{\rho}}}
\int_0^{\infty} u^{r - 1} e^{-\frac{c}{2} u^{\frac{1}{\rho}}}
\, du.
$$
Now the change of variables $ (c/4)^{\rho} u = v $ gives
$$
\int_0^{\infty} u^{r - 1} e^{-\frac{c}{2} u^{\frac{1}{\rho}}} \, du
= \int_0^{\infty} u^{r - 1} e^{-2 (\frac{c^\rho}{4^\rho} u)^{\frac{1}{\rho}}} \, du
= \int_0^{\infty} \frac{4^{r\rho}}{c^{r\rho}} v^{r - 1} e^{-2 v^{\frac{1}{\rho}}} \, dv,
$$
so that
\begin{align*}
|g_r (x)| & \leq \frac{4^{r \rho} M}{c^{r  \rho } (r-1)!}
e^{-\frac{c}{2} x^{\frac{1}{\rho}}} \Big( \int_0^{\infty} e^{- v^{\frac{1}{\rho}}} \, dv \Big) \cdot
\Big( \sup_{v>0} v^{r- 1} e^{-v^{\frac{1}{\rho}}} \Big)
\\
& \leq C M\left(\frac{4^{\rho}\rho^{\rho}}{c^{\rho}}\right)^r (r-1)!^{\rho - 1} e^{-\frac{c}{2}
x^{\frac{1}{\rho}}},
\end{align*}
where we used that the growth order of $ \sup_{v} v^{r - 1}e ^{-v^{\frac{1}{\rho}}}$ is dominated by a constant multiple of $ \rho^{\rho r} (r-1)!^{\rho}$ (combine e.g. \cite[Inequality (3), p.~170]{GS} with Stirling's formula).

Assume now that the result holds for any dimension up to $d$ and let $g(y,y_{d+1})$ be a continuous function in the variables $(y,y_{d+1})\in\mathbb{R}^{d+1}$ satisfying the hypotheses of the lemma. For each $0\leq j\leq r$, define $g_{j}(y)=\int_{-\infty}^{\infty}w^{j}g(y,w)\: dw$, $y\in\mathbb{R}^{d}$.
Applying the induction hypothesis to each function $g_{j}$, we find functions $g_{\beta,j}\in C^{r-j}(\mathbb{R}^{d})$ such that
$g_{j}=\sum_{|\beta|=r-j}\partial^{\beta} g_{\beta,j}$,
$\int_{\mathbb{R}^{d}}g_{\beta,j}(y)dy=0$, and
\begin{align*}
 |g_{\beta,j} (y)| & \leq C_{d,\rho} (r-j)!^{\rho-1} h_{1}^{r-j}
e^{-\frac{c }{2^{d+1}}|y|^{\frac{1}{\rho}}}\sup_{x\in\mathbb{R}^{d}}e^{\frac{c}{2}|x|^{\frac{1}{\rho}}} \int_{-\infty}^{\infty}|w^{j}g(x,w)|\: dw
\\
 & \leq C_{d,\rho} (r-j)!^{\rho-1} h_{1}^{r-j} e^{-\frac{c }{2^{d+1}}|y|^{\frac{1}{\rho}}}
 \\
 &
 \qquad  \quad \quad  \times \left( \sup_{x \in\mathbb{R}^{d}}e^{\frac{c}{2}|x|^{\frac{1}{\rho}}} \int_{-\infty}^{\infty}|w|^{j} e^{-c|(x,w)|^{\frac{1}{\rho}}}\: dw\right) \left(\sup_{u\in\mathbb{R}^{d+1}}|g(u)|e^{c|u|^{\frac{1}{\rho}}}\right)
\\
 &
  \leq 2 C_{d,\rho} (r-j)!^{\rho-1} h_1^{r-j} e^{-\frac{c }{2^{d+1}}|y|^{\frac{1}{\rho}}} \left( \int_{0}^{\infty}w^{j} e^{-\frac{c}{2}w^{\frac{1}{\rho}}}\: dw\right)  \left(\sup_{u\in\mathbb{R}^{d+1}}|g(u)|e^{c|u|^{\frac{1}{\rho}}}\right)
\\
 &
 \leq C_{d+1,\rho} (r-j)!^{\rho-1} j!^{\rho} h^{r-j}
e^{-\frac{c}{2^{d+1}}|y|^{\frac{1}{\rho}}}\sup_{u\in\mathbb{R}^{d+1}}|g(u)|e^{c|u|^{\frac{1}{\rho}}}, \qquad y\in\mathbb{R}^{d},
\end{align*}
for some positive absolute constants $C_{d+1,\rho}$ and $h=h_{d+1,\rho,c}$.

Let $1<\rho'<\rho$ and let
$\eta\in\mathcal{S}^{0}_{\rho'}(\mathbb{R})=\mathcal{F}(\mathcal{D}^{\rho'}(\mathbb{R}))$ be such that
\begin{equation} \label{moments}
\int_{-\infty}^{\infty} \eta(u)du=1 \;\;\; \text{ and} \;\;\; \int_{-\infty}^{\infty}u^{m}\eta(u)du=0,
\end{equation}
for all positive integers $m$. Set
$$
\ell(y,y_{d+1})= g(y,y_{d+1})-\sum_{j=0}^{r}\frac{(-1)^{j}}{j!}\eta^{(j)}(y_{d+1})g_{j}(y).
$$
If $m=j$, after applying $m$-times integration by parts we obtain $\int_{-\infty}^{\infty} u^{m}\eta^{(j)}(u)\:du=(-1)^{m} m!$, and otherwise
$\int_{-\infty}^{\infty} u^{m}\eta^{(j)}(u)\:du=0$ whenever $m \neq j$, due to \eqref{moments}. Therefore, we conclude that
$\int_{-\infty}^{\infty}\ell(y,u)u^{j}\: du=0$ for each $0\leq j\leq r$. The one dimensional case gives a function $\tilde{g}\in C(\mathbb{R}^{d+1}) $ such that $\ell(y,y_{d+1})=\partial^{r}_{y_{d+1}} \tilde{g}(y,y_{d+1})$, $\int_{-\infty}^{\infty}\tilde{g}(y,u)\:du=0$, and (uniformly in $y\in\mathbb{R}^{d}$)
$$
|\tilde{g}(y,y_{d+1})|\leq \tilde{C}_{d+1,\rho} r!^{\rho-1} \tilde{h} ^{r}
e^{-\frac{c }{4}|(y,y_{d+1})|^{\frac{1}{\rho}}}\sup_{u\in\mathbb{R}^{d+1}}|g(u)|e^{c|u|^{1/\rho}}
, \qquad y\in\mathbb{R}^{d},
$$
for some positive absolute constants $\tilde{C}_{d+1,\rho}$ and $\tilde{h}=\tilde{h}_{d+1,\rho,c}$. We obtain the representation
\begin{align*}
g(y,y_{d+1})=\partial^{r}_{y_{d+1}}& (\tilde{g}(y,y_{d+1})+\frac{(-1)^{r}}{r!}g_{r}(y) \eta(y_{d+1}))
\\
&+ \sum_{j=0}^{r-1}\underset{\beta\in \mathbb{N}^{d}}{\sum_{|\beta|=r-j}}\partial^{(\beta,j)}_{(y,y_{d+1})} \left( \frac{(-1)^{j}}{j!}g_{\beta,j}(y)\eta(y_{d+1})\right),
\end{align*}
whence the result in the $d+1$ dimensional case immediately follows.

\end{proof}

We are ready to show Theorem \ref{th conv MRA rho regular}.

\begin{proof}[Proof of Theorem \ref{th conv MRA rho regular}] For (i), note that $q_m\varphi\to\varphi$ pointwise  (actually, the pointwise convergence holds under rather mild hypotheses even for a large class of distributions, cf. \cite{KosVin}). Thus, using Lemma \ref{conv1}, we only have to check that $\{q_m\varphi \}_{m\in\mathbb{N}}$ is a bounded sequence in $\mathcal{S}^{s}_{t} (\mathbb{R}^d)$.

Let $\eta\in \mathcal{S}^{\rho_1}_{\rho_2}(\mathbb{R}^{d})$ be an even function such that $\int_{\mathbb{R}^{d}}\eta(x)\:dx=1.$ We compare each orthogonal projection operator $q_m$ with the convolution operator with kernel $\eta_{m}(x)=2^{md}\eta(2^m x)$. For each $\alpha$, set
$$
R^{\alpha} (x,y) = \partial^{\alpha} _x q_0 (x,y) - \partial^{\alpha} _x  \eta (x-y).
$$
By Lemma \ref{projE} and since $\eta\in \mathcal{S}^{\rho_1}_{\rho_2}(\mathbb{R}^{d})$, the functions $R^{\alpha}(x,x+y)=\partial^{\alpha} _x q_0 (x,y+x) - \partial^{\alpha}  \eta (y)$ satisfy estimates (uniformly in  $x$)
$$
\sup_{y\in\mathbb{R}^{d}} e^{c_1 |y|^{\frac{1} {\rho_2}}}R^{\alpha}(x,x+y) \lesssim h_1^{|\alpha|} \alpha!^{\rho_1}.
$$
Since all moments of $g(y)=R^{\alpha}(x,x+y)$ vanish up to order $|\alpha|$ (cf. \cite[Corollary, p.~40]{Mey}),
Lemma \ref{LemmaIntegrationMRA} applies to yield the existence of kernel functions $R^{\alpha,\beta}\in C(\mathbb{R}^{d}\times \mathbb{R}^{d})$, $|\beta|=|\alpha|$, such that
\begin{equation} \label{Rprop3}
|R^{\alpha, \beta}(x, y)| \lesssim h_2^{|\alpha|} \,
\alpha!^{\rho_1+\rho_2-1}  e^{-c_2 |x - y|^{\frac{1}{\rho_2}}}, \qquad x,y\in\mathbb{R}^{d}, \ \alpha\in\mathbb{N}^{d},
\end{equation}
and

$$
\partial^{\alpha} \circ q_0 = \mathcal{C} \circ \partial^{\alpha} +\sum_{|\beta|=|\alpha|}\mathcal{R}^{\alpha,\beta} \circ \partial^{\alpha},
$$
where $\mathcal{C} $ denotes the operation of convolution by $\eta$ and $\mathcal{R}^{\alpha,\beta}$ is the integral operator with kernel $R^{\alpha,\beta}.$
We thus conclude that, for each $m\in\mathbb{N}$ and each $\varphi\in  \mathcal{S}^{s - \sigma}_{t} (\mathbb{R}^d)$,
\begin{align}
\label{MRA eq twisting}
\partial^{\alpha}_x (q_m \varphi) (x) =& 2^{m d} \int_{\mathbb{R}^d} \eta(2^m (x - y))\partial^{\alpha}_y \varphi(y) \, dy
\\
& \nonumber \qquad+2^{m d}
 \sum_{|\beta| = |\alpha|} \int_{\mathbb{R}^d}  R^{\alpha,
\beta} (2^m x, 2^m y) \partial^{\alpha}_y \varphi(y) \, dy.
\end{align}

We now estimate the second term of \eqref{MRA eq twisting}.
Using \eqref{Rprop3}, we have, for some $h,c>0$,
\begin{align*}
2^{md}\left|\int_{\mathbb{R}^d}  R^{\alpha,
\beta} (2^m x, 2^m y) \partial^{\alpha}_y \varphi(y) \, dy\right|
 &\lesssim (h h_2)^{|\alpha|} \alpha!^{s} 2^{md}\int_{\mathbb{R}^d}
e^{- c_2 |2^{m}(x - y)|^{\frac{1}{\rho_2}}} e^{-c |y|^{\frac{1}{t}}}
\, dy
\\
& \lesssim (h h_2)^{|\alpha|} \alpha!^{s}2^{md} e^{-c
|x|^{\frac{1}{t}}} \int_{\mathbb{R}^d} e^{- 2^{\frac{m}{\rho_2}} c_2
|u|^{\frac{1}{\rho_2}}} e^{c |u|^{\frac{1}{\rho_2}}} \, du
\\
&
\lesssim (h h_2)^{|\alpha|} \alpha!^{s} e^{-c |x|^{\frac{1}{t}}} \int_{\mathbb{R}^d} e^{-  c_2 |s|^{\frac{1}{\rho_2}}} e^{c2^{-\frac{m}{\rho_2}} |s|^{\frac{1}{\rho_2}}} \, ds
\\
& \lesssim (h h_2)^{|\alpha|}\alpha!^{s} e^{-c |x|^{\frac{1}{t}}}
\int_{\mathbb{R}^d} e^{- \frac{c_2}{2} |s|^{\frac{1}{\rho_2}}} \,
ds,
\end{align*}
when $ m $ is chosen such that $ 2^{m } \geq (2c/c_2)^{\rho_2}. $ Likewise, we verify that  $\{\eta_m\ast \varphi \}_{m\in\mathbb{N}}$ is a bounded sequence in $\mathcal{S}^{s}_{t}(\mathbb{R}^{d})$. For some $c_3>0$ (and the same $h,c$ as above),
\begin{align*}
2^{md}\left|\int_{\mathbb{R}^d}  \eta (2^m (x -y)) \partial^{\alpha}_y \varphi(y) \, dy\right|
 &\lesssim h^{|\alpha|} \alpha!^{s} 2^{md}\int_{\mathbb{R}^d}
e^{- c_3 |2^{m}(x - y)|^{\frac{1}{\rho_2}}} e^{-c |y|^{\frac{1}{t}}}
\, dy
\\
& \lesssim h^{|\alpha|}\alpha!^{s} e^{-c |x|^{\frac{1}{t}}}
\int_{\mathbb{R}^d} e^{- \frac{c_3}{2} |s|^{\frac{1}{\rho_2}}} \,
ds,
\end{align*}
whenever $2^{m } \geq (2c/c_3)^{\rho_2}$.
This shows the boundedness of $ \{ q_m \varphi \}_{m \in \mathbb{N}} $ in $\mathcal{S}^{s}_{t}(\mathbb{R}^{d})$ and we have therefore proved the first part of the theorem.

Let us now show part (ii). Let $\varphi\in \mathcal{S}^{s-\sigma}_{t} (\mathbb{R}^d)$.
We note that
$$
\langle q_m f, \varphi \rangle =  \langle  f, q_m \varphi \rangle,
$$
so that the result follows from part (i).
\end{proof}

\section{Wavelet expansions in Gelfand-Shilov spaces} \label{SecWaveletGS}

In this section we study wavelet expansions of Gelfand-Shilov functions and ultradistributions.  The automatic vanishing moment property of highly regular wavelets naturally leads to work with the spaces $(\mathcal{S}^{s}_{t})_0(\mathbb{R}^{d})$ if one intends to develop a wavelet expansion theory for their duals (cf. \cite{SanVin} for the tempered distribution counterpart).

We are interested in the multidimensional case. For the construction of wavelet bases of $L^{2}(\mathbb{R}^{d})$, we shall follow the tensor product approach based on a single one-dimensional MRA wavelet (see \cite[Chapter 3]{Mey}), which we now briefly recall.

Let $\psi \in (\mathcal{S}^{\rho_1}_{\rho_2})_{0}(\mathbb{R})$ be a $(\rho_1,\rho_2)$-regular orthonormal wavelet, with scaling function $\phi\in \mathcal{S}^{\rho_1}_{\rho_2}(\mathbb{R})$. Set $Q = \{ 0,1\}^d \setminus (0,\dots,0) $ and $ \Lambda =  Q \times \mathbb{Z} \times \mathbb{Z}^d$. We write $ \psi_0 (u) = \phi (u) $ and $ \psi_1 (u) = \psi (u), $ $ u
\in \mathbb{R}$. Let
$
\psi_\epsilon (x) = \psi_{\epsilon_1} (x_1) \psi_{\epsilon_2}
(x_2) \cdots \psi_{\epsilon_d} (x_d),$ $x= (x_1, x_2, \dots,
x_d ) \in \mathbb{R}^d,$
where $ \epsilon \in Q$, which
gives $ 2^d - 1 $ functions $\psi_{\epsilon} .$ Finally, for each $\lambda=(\epsilon,m,n)\in \Lambda$, we set
\begin{eqnarray*}
\psi_{\lambda}(x)=\psi_{\epsilon, m, n} (x) & =  &2^{md/2} \psi_{\epsilon} (2^m x - n), \qquad x\in\mathbb{R}^{d}.
\end{eqnarray*}
Note that each $\psi_{\lambda} \in  (\mathcal{S}^{\rho_1}_{\rho_2})_{0}(\mathbb{R}^d)$ by the construction.
Then,  it follows \cite{Mey} that the collection of functions
\begin{equation*}
\{ \psi_{\lambda} \; | \; \lambda \in \Lambda\} =
 \{ \psi_{\epsilon, m, n} \; | \; \epsilon \in Q, m \in \mathbb{Z},
n \in \mathbb{Z}^d \}
\end{equation*}
is an orthonormal basis of $ L^2 (\mathbb{R}^d). $ Every function $f\in L^{2}(\mathbb{R}^{d})$ can then be expanded as
\begin{equation}
\label{wexp eq 1}
f = \sum_{\lambda\in\Lambda} c^{\psi}_{\lambda}(f) \psi_{\lambda},
\end{equation}
where the wavelet coefficients of $f$ can be expressed in terms of the wavelet transform,
\begin{equation}
\label{coeffwavtrans}
c^{\psi}_{\lambda} (f)=c^{\psi}_{\epsilon, m, n} (f) = \langle f,\bar{\psi}_{\lambda}\rangle= 2^{-\frac{md}{2}} \, \mathcal{W}_{\psi} f (n
2^{-m}, 2^{-m}), \qquad \lambda=(\epsilon,m,n)\in\Lambda.
\end{equation}
These wavelet coefficients remain well-defined for any $f\in (\mathcal{S}^{\rho_1}_{\rho_2})'_{0}(\mathbb{R}^{d})$ via the above formulas.

We can now state our main result from this section.

\begin{theorem} \label{th convergence} Let $\psi\in (\mathcal{S}_{\rho_2}^{\rho_1})_{0}(\mathbb{R})$ be a $(\rho_1,\rho_2)$-regular orthonormal wavelet, where $ \rho_1 \geq 0$ and $\rho_2 > 1.$ Set
 $ \sigma = \rho_1 + \rho_2 - 1$ and consider parameters $ s > \sigma$ and $ t > \sigma + 1.$
\begin{itemize}
\item [(i)] If $ \varphi \in
(\mathcal{S}^{s - \sigma}_{t - \sigma})_0 (\mathbb{R}^d), $ then
$$
\varphi = \sum_{\lambda \in \Lambda}  c_{\lambda}^{\psi} (\varphi) \,
\psi_{\lambda} \qquad\mbox{converges in }(\mathcal{S}^s_t)_0 (\mathbb{R}^d).
$$

\item [(ii)] If $ f \in
((\mathcal{S}^s_t)_0 (\mathbb{R}^d))' $, then its wavelet series expansion \eqref{wexp eq 1}
converges in (the strong dual topology of) $ ((\mathcal{S}^{s-\sigma}_{t-\sigma})_0
(\mathbb{R}^d))'. $
\item [(iii)] We have the Parseval identity
$$\langle f, \varphi \rangle = \sum_{\lambda \in \Lambda} c^{\psi}_{\lambda} (f)
\, c^{\bar{\psi}}_{\lambda} (\varphi),$$
for any $ f \in
((\mathcal{S}^s_t)_0 (\mathbb{R}^d))' $ and $ \varphi \in
(\mathcal{S}^{s - \sigma}_{t - \sigma})_0 (\mathbb{R}^d).
$
\end{itemize}
\end{theorem}

We set the ground for the proof of Theorem \ref{th convergence} by giving some growth estimates for wavelet coefficients of Gelfand-Shilov functions and ultradistributions.
To quantify the decay of wavelet coefficients it is convenient to introduce spaces of rapidly decreasing multi-sequences
$ \mathcal{W}^{s}_{t,\rho_1, \rho_2} (\Lambda) $ as follows:
$ \{c_{\epsilon, m, n}\}_{(\epsilon, m, n) \in \Lambda}  = \{c_{\lambda}\}_{\lambda \in \Lambda} \in \mathcal{W}^{s}_{t,\rho_1, \rho_2} (\Lambda) $
if and only if there exists $ k \in \mathbb{N} $ such that the norm
\begin{equation} \label{sW}
\| \{c_{\lambda}\} \| ^{\mathcal{W}^{s}_{t,\rho_1, \rho_2}}_k :=
\sup_{\lambda \in \Lambda} |c_{\lambda}|
 e^{k \big( (\frac{1}{2^m})^{\frac{1}{t -
{\rho_2}}} + (2^m)^{\frac{1}{s - {\rho_1}}} +
|\frac{n}{2^m}|^{\frac{1}{t}} \big)}
\end{equation}
is finite. Then, $ \mathcal{W}^{s}_{t,\rho_1, \rho_2} (\Lambda) $ becomes a
(DFS)-space with the inductive limit topology induced by the norms (\ref{sW}). Its dual space $ (\mathcal{W}^{s}_{t,\rho_1, \rho_2} (\Lambda))' $ is the (FS)-space consisting of multisequences $ \{c_{\lambda }\} $ such that
$$
\| \{c _{\lambda}\} \|^{(\mathcal{W}^{s}_{t,\rho_1, \rho_2})'}_{-k} :=
\sup_{\lambda \in \Lambda} |c_{\lambda}|
 e^{-k \big( (\frac{1}{2^m})^{\frac{1}{t -
{\rho_2}}} + (2^m)^{\frac{1}{s - {\rho_1}}} +
|\frac{n}{2^m}|^{\frac{1}{t}} \big)}
$$
is finite for every $ k \in \mathbb{N}. $

\begin{lemma} \label{coef} Let $\psi\in (\mathcal{S}_{\rho_2}^{\rho_1})_{0}(\mathbb{R})$ be a $(\rho_1,\rho_2)$-regular orthonormal wavelet, where $ \rho_1 \geq 0$ and $\rho_2 > 1.$ Set
$ \sigma = \rho_1 + \rho_2 - 1$ and consider parameters $ s > \sigma$ and $ t > \sigma + 1.$ Then, for every $l>0$ there is $k>0$ such that, for every $\varphi \in (\mathcal{S}^{s - \sigma}_{t - \sigma})_0
(\mathbb{R}^d) $ with $p_l^{s - \sigma, t - \sigma} (\varphi)<\infty$,

\begin{equation*}
\|\{ c_{\lambda} ^{\psi} (\varphi)\}\| ^{\mathcal{W}^{s}_{t,\rho_1, \rho_2}}_k
\lesssim  p_l^{s - \sigma, t - \sigma} (\varphi).
\end{equation*}
\end{lemma}

\begin{proof}
By Theorem \ref{th3} with $ \tau_1 = t- \rho_2 $ and $ \tau_2 = s
- \rho_1, $ we have that
$$
{\mathcal W}_{\psi_{\epsilon}} : ({\mathcal S}^{s - \sigma}_{t -
\sigma})_0 (\mathbb{R}^d) \to {\mathcal S}^{s}_{t, t- \rho_2, s -
\rho_1} (\mathbb{H}^{d+1}), \qquad \epsilon \in Q,
$$
is a continuous mapping, which combined with the formula \eqref{coeffwavtrans} obviously yields the assertion.
\end{proof}

We now present the proof of Theorem \ref{th convergence}.

\begin{proof}[Proof of Theorem \ref{th convergence}] Let us prove part (i). Set $ \Lambda_{M,N} = \{ (\epsilon, m, n) \in \Lambda \; : \; |m|  \leq M, \; |n| \leq N \}$. We already know that the wavelet series converges pointwise to  $\varphi$ (actually, the convergence even holds in the Schwartz space, cf. e.g. \cite{SanVin}), so that we may write the difference between $\varphi$ and a partial sum of the wavelet expansion as a pointwise convergent tail series. Our task therefore consists in showing (cf. \eqref{norma-GS2}) that there
exist $ h, c > 0 $ such that
\[
 \lim_{M, N \to \infty} p^{s, t}_{h, c} \Big(\sum_{\lambda \in \Lambda \setminus \Lambda_{M,N}}
c_{\lambda}^{\psi} (\varphi) \, \psi_{\lambda} \Big)
 \leq \lim_{M, N \to \infty} \sum_{\lambda \in \Lambda \setminus  \Lambda_{M,N}} |c^{\psi}_{\lambda}(\varphi)| \sup_{x\in\mathbb{R}^{d},\: \beta\in\mathbb{N}^{d}} S^{\beta}_{\lambda}(x)= 0,
\]
where
$$
S_{\lambda}^{\beta} (x) = S_{\epsilon, m, n}^{\beta} (x)
 =\frac{h^{|\beta|}}{\beta!^s} e^{c
|x|^{\frac{1}{t}}}|\partial^\beta \psi_{\lambda} ( x)|
= 2^{m(\frac{ d}{2}+|\beta|)} \frac{h^{|\beta|}}{\beta!^s} e^{c
|x|^{\frac{1}{t}}}|\partial^\beta_x \psi_{\epsilon} (2^m x - n)|.
$$
Lemma \ref{coef} implies that there is $ k > 0
$ such that
$$
|c^{\psi}_{(\epsilon,m,n)} (\varphi)| \lesssim e^{-k \big(
(\frac{1}{2^m})^{\frac{1}{t - {\rho_2}}} + (2^m)^{\frac{1}{s -
\rho_1}} + |\frac{n}{2^m}|^{\frac{1}{t}} \big)}, \qquad  (\epsilon,m,n) \in \Lambda.
$$
Since $ \psi_{\lambda} \in
(\mathcal{S}^{\rho_1}_{\rho_2})_0 (\mathbb{R}^d) $, there are $ h_1, c_1 > 0 $ such that
\begin{align*}
S_{\epsilon,m,n}^{\beta} (x)
& \lesssim \frac{h^{|\beta|}}{\beta!^s} 2^{m (\frac{ d}{2}+|\beta|)} \, e^{c |x|^{\frac{1}{t}}} \,
\frac{\beta!^{\rho_1}}{h_1^{|\beta|}} \, e^{-c_1
|2^{m}x-n|^{\frac{1}{\rho_2}}}
\\
& \lesssim \left(\frac{h}{h_1}\right)^{|\beta|}
\frac{2^{m (\frac{ d}{2}+|\beta|)}}{\beta!^{s - \rho_1}} \, e^{c
|\frac{n}{2^m}|^{\frac{1}{t}}} \, e^{c
|x-\frac{n}{2^m}|^{\frac{1}{t}} - c_1 |2^{m}x-n|^{\frac{1}{\rho_2}}}.
\end{align*}

We choose $c=\min\{c_1,k/2\}$ and make use of $\rho_2<t$.  The function $ g(r):= c \, \big(2^{-m}r\big)^{\frac{1}{t}} - c_1 \, r^{\frac{1}{\rho_2}}$ attains its maximum value on $[0,\infty)$ at $ r_0 = \big(\frac{c \: \rho_2}{c_1 \: t \: 2^{\frac{m}{t}}}\big)^{\frac{\rho_2 \: t}{t - \rho_2}} $ and
$$
g(r_0)< c 2^{-\frac{m}{t}}\left(\frac{c \rho_2}{ c_1  t 2^{\frac{m}{t}}}\right)^{\frac{\rho_2}{t - \rho_2}}< c2^{-\frac{m}{t-\rho_2}} .
$$
Hence,
\[
 \sup_{x\in\mathbb{R}^{d}}S_{\epsilon,m,n}^{\beta} (x) \lesssim  \left(\frac{h}{h_1}\right)^{|\beta|} \frac{2^{m (\frac{ d}{2}+|\beta|)}}{\beta!^{s - \rho_1}}  \, e^{\frac{k}{2} |\frac{n}{2^m}|^{\frac{1}{t}}}
 e^{\frac{k}{2} (\frac{1}{2^m})^{\frac{1}{t - \rho_2}}}.
\]
From $ r^{\nu n} \leq n!^{\nu} e^{\nu r}$, $(n + m)! \leq 2^{n + m} n! m!$, and $|\beta|! \leq d^{|\beta|} \beta! $,
$\nu,r > 0,$ $n,m \in \mathbb{N}, $ $\beta \in \mathbb{N}^d $ (see e.g. \cite[pp.~13--14]{NR}),
it follows that
\begin{align*}
2^{m (\frac{ d}{2}+|\beta|)} & = 2^{\frac{ md}{2}}\left(\frac{2(s - \rho_1)}{c}\right)^{|\beta| (s -
\rho_1)} \left(\frac{c \: 2^{\frac{m}{s - \rho_1}}}{2(s - \rho_1)}
\right)^{|\beta| (s - \rho_1)} \\
& \leq
2^{\frac{ md}{2}}\left(\frac{2(s - \rho_1)}{c}\right)^{|\beta| (s -
\rho_1)}
|\beta|! ^{s - \rho_1} e^{\frac{c}{2} \: (2^m)^\frac{1}{s - \rho_1}}\\
&
\leq
\left(\frac{2d(s - \rho_1)}{c}\right)^{|\beta| (s -
\rho_1)}
\beta! ^{s - \rho_1} 2^{\frac{ md}{2}} e^{\frac{c}{2} \: (2^m)^\frac{1}{s - \rho_1}}\\
&
\lesssim
\left(\frac{2d(s - \rho_1)}{c}\right)^{|\beta| (s -
\rho_1)}
\, \beta!^{s - \rho_1} \, e^{c \: (2^m)^\frac{1}{s
- \rho_1}}.
\end{align*}
 Choosing $h=h_{1}(2d(s - \rho_1)/c)^{-|\beta| (s -
\rho_1)}$, we finally obtain
$$
|c^{\psi}_{(\epsilon,m,n)} (\varphi)| \sup_{x\in  \mathbb{R}^{d},\: \beta \in
\mathbb{N}^d}  S_{(\epsilon,m,n)}^{\beta} (x) \lesssim  e^{-\frac{k}{2} (2^m)^{\frac{1}{s - \rho_1}}}
e^{-\frac{k}{2} |\frac{n}{2^m}|^{\frac{1}{t}}} e^{-\frac{k}{2}
(\frac{1}{2^m})^{\frac{1}{t - \rho_2}}},
$$
which establishes part (i) of the theorem.

We now show parts (ii) and (iii) simultaneously. Since the Gelfand-Shilov spaces are barreled, it suffices to show the weak convergence of \eqref{wexp eq 1}.
Let $ \varphi \in (\mathcal{S}^{s-\sigma}_{t-\sigma})_0 (\mathbb{R}^d). $
Since $\bar{\psi}$ is also a $(\rho_1,\rho_2)$-regular orthonormal wavelet, part (i) yields
$$
\varphi = \sum_{\lambda \in \Lambda} c_{\lambda}^{\bar{\psi}} (\varphi) \,
\bar{\psi}_{\lambda},
$$
with convergence in $ (\mathcal{S}^s_t)_0 (\mathbb{R}^d). $ By assumption  $ f \in
((\mathcal{S}^s_t)_0 (\mathbb{R}^d))'\subset ((\mathcal{S}^{s-\sigma}_{t-\sigma})_0
(\mathbb{R}^d))', $ so that we may exchange summation with dual pairing in

\begin{align*}
\langle f, \varphi \rangle & = \Big\langle f, \sum_{\lambda \in \Lambda}
 c_{\lambda}^{\bar{\psi}} (\varphi) \bar{\psi}_{\lambda} \Big\rangle =
\sum_{\lambda \in \Lambda} c_{\lambda}^{\bar{\psi}} (\varphi) \langle f,
\bar{\psi}_{\lambda} \rangle  \qquad \Big(= \sum_{\lambda \in \Lambda} c^{\psi}_{\lambda} (f)
\, c^{\bar{\psi}}_{\lambda} (\varphi)\Big)
\\
& = \sum_{\lambda \in \Lambda} c_{\lambda}^{\psi} (f) \langle \psi_{\lambda}, \varphi  \rangle  = \Big\langle \sum_{\lambda \in \Lambda} c_{\lambda}^{\psi} (f)
\psi_{\lambda}, \varphi \Big\rangle.
\end{align*}
and the theorem has been proved.
\end{proof}

We end this section with a corollary. For functions and ultradistributions, we consider the wavelet coefficient mapping $c^{\psi}:f \mapsto \{c^{\psi}_{\lambda}(f)\}_{\lambda\in\Lambda}$.

\begin{corollary} \label{iso} Let $\psi\in (\mathcal{S}_{\rho_2}^{\rho_1})_{0}(\mathbb{R})$ be a $(\rho_1,\rho_2)$-regular orthonormal wavelet, where $ \rho_1 \geq 0$ and $\rho_2 > 1.$ Set
 $ \sigma = \rho_1 + \rho_2 - 1$ and consider $ s > \sigma$ and $ t > \sigma + 1.$ Then, the following mappings are continuous and injective:
\begin{equation}
\label{wcoefeq1}
 c^{\psi} : (\mathcal{S}^{s - \sigma}_{t -
\sigma})_0 (\mathbb{R}^d) \to \mathcal{W}^{s}_{t,\rho_1,
\rho_2} (\Lambda)
\end{equation}
and
\begin{equation}
\label{wcoefeq2}
c^{\psi} : ((\mathcal{S}^s_t)_0 (\mathbb{R}^d))'
\to (\mathcal{W}^{s}_{t,\rho_1, \rho_2} (\Lambda))' .
\end{equation}
\end{corollary}

\begin{proof}  Theorem \ref{th convergence} yields the injectivity of both mappings. Furthermore, the continuity of \eqref{wcoefeq1} is a direct consequence of Lemma \ref{coef}.
By Proposition \ref{Calderon 2} and (\ref{coeffwavtrans}), the mapping \eqref{wcoefeq2} is bounded, which is equivalent to continuity for Fr\'{e}chet spaces.
\end{proof}

\end{document}